\documentclass[draft,reqno]{amsart}

\thanks{Version of \today.}

\title{Fields with several commuting derivations}
\author{David Pierce}
\address{Mathematics Department\\
Mimar Sinan Fine Arts University\\
Istanbul, Turkey}

\email{dpierce@msgsu.edu.tr}
\urladdr{http://mat.msgsu.edu.tr/~dpierce}

\usepackage{amscd}             % commutative diagram
\usepackage{pstricks,pst-node}
\usepackage[mathscr]{euscript}
\usepackage{stmaryrd}  % for \trianglelefteqslant
\usepackage{bm}
\usepackage[all]{xy}
\usepackage{verbatim}

\usepackage{amssymb}
\renewcommand{\leq}{\leqslant}
\renewcommand{\geq}{\geqslant}
\renewcommand{\phi}{\varphi}

\renewcommand{\land}{\mathrel{\&}}

\newcommand{\myboxed}[2]{$#1$#2} % for symbolism (#1) being
\newcommand{\defn}[2]{\textbf{#1#2}}   % for terms (#1) being defined;
\newtheorem{theorem}{Theorem}[section]
\newtheorem{lemma}[theorem]{Lemma}
\newtheorem{corollary}[theorem]{Corollary}

\newcommand{\qsep}{\;}                 % follows a quantified variable
\newcommand{\Forall}[1]{\forall{#1}\qsep }
\newcommand{\Exists}[1]{\exists{#1}\qsep }
\newcommand{\Iff}{\iff}
\newcommand{\tuple}[1]{\bm{#1}}
\newcommand{\lto}{\Rightarrow}

\newcommand{\str}[1]{\mathfrak{#1}}            % structure
\newcommand{\diag}[1]{\operatorname{diag}(#1)} % (Robinson) diagram of
\newcommand{\robfn}[1]{\widehat{#1}}   % ``Robinson function''
\newcommand{\Th}[1]{\operatorname{Th}(#1)}     % theory of
\newcommand{\Mod}[1]{\operatorname{Mod}(#1)}   % class of models of
\newcommand{\proves}{\vdash}
\newcommand{\thy}[1]{\mathrm{#1}}      % for the following:
\newcommand{\ACF}{\thy{ACF}}           % algebraically closed fields
\newcommand{\SCF}{\thy{SCF}}           % separably closed fields
\newcommand{\DF}{\thy{DF}}             % differential fields
\newcommand{\DPF}{\thy{DPF}}           % differentially perfect fields
\newcommand{\ACFA}{\thy{ACFA}}         % fields with generic autom.
\newcommand{\mDF}{\text{$m$-$\thy{DF}$}} % fields with m derivations
\newcommand{\DCF}{\thy{DCF}}           % differentially closed fields
\newcommand{\mDCF}{\text{$m$-$\thy{DCF}$}}  % differentially closed fields
\newcommand{\included}{\subseteq}      % inclusion
\newcommand{\pincluded}{\subset}       % proper inclusion    
\newcommand{\size}[1]{\lvert#1\rvert}
\usepackage{upgreek}
\newcommand{\vnn}{\upomega}        % von Neumann natural numbers
\newcommand{\tleq}{\trianglelefteqslant} % the total ordering of \vnn^m
\newcommand{\tl}{\vartriangleleft} % the strict total ordering of \vnn^m
\newcommand{\ichar}{\bm i}         % char. fn of {i} in \vnn^m
\newcommand{\jchar}{\bm j}         % char. fn of {j} in \vnn^m
\newcommand{\kchar}{\bm k}         % char. fn of {k} in \vnn^m

\newcommand{\Char}[1]{\operatorname{char}(#1)} % characteristic of a field
\newcommand{\sep}[1]{#1^{\mathrm{sep}}} % separable closure
\newcommand{\alg}[1]{#1^{\mathrm{alg}}} % algebraic closure
\newcommand{\Aff}{\mathbb{A}}       % affine space
\newcommand{\lrv}{E}               % a Lie-ring of vectors
\DeclareMathOperator{\dee}{d}      % derivation into the dual of this

\begin{document}

  \begin{abstract}
  For every natural number $m$,
the existentially closed models of the theory of fields with $m$ commuting
    derivations can be given a first-order geometric characterization in several ways.
    In particular, the theory of these differential fields has a model-companion.
    The axioms are that 
    certain differential varieties
    determined by certain ordinary varieties are nonempty.  There is no restriction on the characteristic of the underlying field.
  \end{abstract}

 \maketitle

% \tableofcontents

How can we tell whether a given system of partial differential equations
has a solution?  An answer given in this paper is that, if we
differentiate the equations enough times, and no contradiction arises,
then it never will, and the system is soluble.  Here, the meaning of
`enough times'
can be expressed \emph{uniformly}; this is one way of showing that the
theory of fields with a given
finite number of commuting derivations has a model-companion.
In fact, this theorem is worked out here (as Corollary~\ref{cor:dcf},
of Theorem~\ref{thm:dcf}), not in terms of polynomials, but 
in terms of the varieties that they define, and the function-fields of
these: in a word, the treatment is \emph{geometric.}

The
theory of fields with 
$m$ commuting derivations will be called here $\mDF$; its model-companion, $\mDCF$.  A specified characteristic can be indicated by a subscript.
The model-companion of $\mDF_0$ (in
characteristic~$0$) has been axiomatized before, explicitly in terms
of differential polynomials: see \S \ref{sect:several}.  The existence
of a model-companion of $\mDF$ (with no specified characteristic) appears to
be a new result when $m>1$ (despite a remark by Saharon Shelah
\cite[p.~315]{MR0344116}: `I am
quite sure that for characteristic $p$ as well, [making $m$ greater
  than $1$] does not make any essential
difference').

The theory of model-companions and model-completions was worked out
decades ago; perhaps for that very reason, it may be worthwhile to
review the theory here, as I do in \S \ref{sect:mt}.  
In \S \ref{sect:one}, I review the various known characterizations of
existentially closed fields with single derivations.
In fact, little of this work is of use in the passage to several derivations; but this near-irrelevance is itself interesting.
In \S \ref{sect:several}, I analyze the error of my earlier
attempt, in \cite{MR2000487}, to axiomatize $\mDCF_0$
in terms of
differential forms.  Something of value from this earlier work does remain: when we do have $\mDCF_0$, or more generally $\mDCF$, then we can obtain from it a model-companion of the theory of fields with $m$ derivations whose linear span over the field is closed under the Lie bracket.
In \S \ref{sect:resolution}, I obtain $\mDCF$ itself.

\section{Model-theoretic background}\label{sect:mt}

I try in this section
to give original references, when
I have been able to consult them.  An exposition can also be found for example in Hodges \cite{MR94e:03002} (particularly chapter 8). 
% The present account will follow the typographical convention whereby a word in \defn{boldface}{} is being defined by the sentence in which it appears.

Let $\Gamma$ be a
set of (first-order) sentences in some signature; the structures that have this signature and are models of $\Gamma$ compose the
class denoted by
\myboxed{\Mod{\Gamma}}.
Every class $\bm K$ of structures of some signature has a \textbf{theory,} denoted by
\myboxed{\Th{\bm K}}; this is the set of sentences in the signature that are true in each of the structures in $\bm K$.  Immediately $\bm K\included\Mod{\Th{\bm K}}$; in case of
equality, $\bm K$ is called \defn{elementary}.   Similarly, $\Gamma\included\Th{\Mod{\Gamma}}$; in case of equality, that is, in case $\Gamma$ is actually the theory of $\Mod{\Gamma}$, then $\Gamma$ is called a \textbf{theory,} simply.  So there is a Galois correspondence between elementary classes and theories.

Let $\str M$ be an arbitrary (first-order) structure; the theory of $\{\str M\}$ is denoted by \myboxed{\Th{\str M}}.
The structure $\str M$ has the universe $M$.  The structure denoted by
$\str M_M$ is the expansion of $\str M$ that has a name for every
element of $M$. 
Then $\str M$ embeds in $\str N$ if and only if $\str M_M$ embeds in
an expansion of $\str N$.  The class of
structures in 
which $\str M$ embeds need not be elementary: for example, $\str M$ could be an uncountable model of a countable theory.
However, the class of
structures in which $\str M_M$ embeds \emph{is} elementary.  The
theory of the latter class is the \defn{diagram}{} of $\str M$, or
\myboxed{\diag{\str M}}: it is
axiomatized by the quantifier-free sentences in $\Th{\str M_M}$
\cite[Thm 2.1.3, p.~24]{MR0153570}. 
A model of $\Th{\str M_M}$ itself is just a structure in
which $\str M_M$
embeds \textbf{elementarily.}  Thus the class of such structures is elementary.
The class of
substructures of models of a theory $T$ is elementary, and its theory is
denoted by \myboxed{T_{\forall}}: this is axiomatized by the universal
sentences of $T$ \cite[Thm 3.3.2, p.~71]{MR0153570}.

By a \defn{system over}{} $\str M$, I mean a finite conjunction of
atomic and negated atomic formulas in the signature of $\str M_M$;
likewise, a system \defn{over}{} a theory $T$ is in the signature
of $T$.  A structure $\str M$ \defn{solves}{} a system $\phi(\tuple x)$
if $\str M\models\Exists{\tuple x}\phi(\tuple x)$.  Note well here that
$\tuple x$, in boldface, is a \emph{tuple} of variables, perhaps
$(x^0,\dots,x^{n-1})$.  By an
\defn{extension}{} of a model of $T$, I mean another model of $T$ of
which the first is a substructure.  Two systems over a model $\str M$
of $T$ will be called \defn{equivalent}{} if they are soluble in precisely the same extensions.

An \defn{existentially closed}{} model of $T$ is a model of $T$ that
solves every system over
itself that is soluble in some extension.  So a model $\str M$ of $T$ is
existentially closed if and only if
$T\cup\diag{\str M}\proves\Th{\str M_M}_{\forall}$, that is, every
extension of $\str M$ is a substructure of an elementary extension
  (\cite[\S 7]{MR0277372} or \cite[\S 2]{MR51:12518}).

A theory is \defn{model-complete}{} if its every
model is existentially closed.  An equivalent formulation explains the
name: $T$ is model-complete if and only if $T\cup\diag{\str M}$ is
complete whenever $\str M\models T$
\cite[Ch.~2]{MR0472504}.  

Suppose every model of $T$ has an existentially closed extension.
Such is the case when $T$ is \defn{inductive}, that is, $\Mod{T}$ is
closed under unions of chains \cite[Thm~7.12]{MR0277372}: equivalently,
$T=T_{\forall\exists}$ \cite{MR0089813,MR0103812}.  Suppose further
that we have a uniform
first-order way to tell when systems over models of $T$ are soluble in
extensions: more precisely, suppose there is a function
\begin{equation}\label{eqn:fn}
\phi(\tuple x,\tuple y)\longmapsto
\robfn{\phi}(\tuple x,\tuple y),
\end{equation}
where 
$\phi(\tuple x,\tuple y)$ ranges over the systems over $T$ (with
variables analyzed as shown),
such that, for every model $\str
M$ of $T$ and every tuple $\tuple a$ of parameters from $M$, the
system $\phi(\tuple x,\tuple a)$ is soluble in some extension of $\str
M$ just in case $\robfn{\phi}(\tuple x,\tuple a)$ is soluble in $\str M$.
Then the existentially closed
models of $T$ compose an elementary class, whose theory $T^*$ is
axiomatized by $T$ together with the sentences
\begin{equation}\label{eqn:sen}
  \Forall{\tuple y}(\Exists{\tuple x}\robfn{\phi}(\tuple x,\tuple
  y)\lto
\Exists{\tuple x}\phi(\tuple x,\tuple y)).
\end{equation}
Immediately, $T^*$ is model-complete, so $T^*\cup\diag{\str M}$ is
complete when $\str M\models T^*$.  What is more,
$T^*\cup\diag{\str M}$ is complete whenever $\str M\models
T$ \cite[Thm 5.5.1]{MR0153570}.

In general, $T^*$ is a \defn{model-completion}{} of $T$ if 
$T^*{}_{\forall}\included T\included T^*$
and 
$T^*\cup\diag{\str M}$ is
complete whenever $\str M\models T$.  Model-completions are unique
\cite[(2.8)]{MR0091922}.   We have sketched the proof of part of the
following; the rest is \cite[(3.5)]{MR0091922}.

\begin{lemma}[Robinson's Criterion]
\label{lem:rob-crit} 
  Let $T$ be inductive.  Then $T$ has a model-comple\-tion if and
  only if a function $\phi(\tuple x,\tuple
  y)\mapsto\robfn{\phi}(\tuple x,\tuple y)$ exists as
  in~\eqref{eqn:fn}. 
In this case, the model-completion is axiomatized \emph{modulo} $T$ by
  the sentences in~\eqref{eqn:sen}.
\end{lemma}

If $T_{\forall}=T^*{}_{\forall}$ and $T^*$ is model-complete, then
$T^*$ is a \defn{model-companion}{} of~$T$
(\cite[\S 5]{MR0272613}; \emph{cf.}~\cite[\S 2]{MR0277372}).
Model-completions are model-companions, and model-companions are
unique \cite[Thm 5.3]{MR0272613}.
If $T$ has a
model-companion, then its models are just the existentially closed
models of $T$ \cite[Prop.~7.10]{MR0277372}.
Conversely, if $T$ is inductive,
and the class of existentially closed models 
of $T$ is elementary, then the theory of this class is the
model-companion of $T$ \cite[Cor.~7.13]{MR0277372}.

\section{Fields with one derivation}\label{sect:one}

Let \myboxed{\DF}{} be the theory of fields with a derivation $D$, and let
\myboxed{\DPF}{} be 
the theory of models of $\DF$ that, for each prime $\ell$, satisfy also
\begin{equation*}
  \Forall x\Exists y(\ell=0\land
  Dx=0\lto y^{\ell}=x),
\end{equation*}
where the first occurrence of $\ell$ stands for $1+\dots+1$ (with $\ell$ occurrences of $1$); and $y^{\ell}$ stands for $y\dotsm y$ (with $\ell$ occurrences of $y$).
Models of $\DPF$ are called \defn{differentially perfect}.  A
subscript on the name of one of these theories will indicate a 
required characteristic for the field.  In particular, we have $\DPF_0$,
which is the same as $\DF_0$. 

Abraham Seidenberg
\cite{MR0082487} shows the existence of the function
in Lemma~\ref{lem:rob-crit} in case
$T$ is $\DPF_p$, where $p$ is prime or $0$.  Consequently:

\begin{theorem}[Robinson]
  $\DF_0$ has a model-completion, called \myboxed{\DCF_0}.
\end{theorem}

\begin{theorem}[Wood \cite{MR48:8227}]
  If $p$ is prime, then $\DF_p$ has a model-companion, called
  \myboxed{\DCF_p}, which is the model-completion of $\DPF_p$. 
\end{theorem}

The existence of a model-companion or model-completion of a theory does not
necessarily tell us much about the existentially closed models of the
theory.  Since it involves \emph{all} systems over a given theory,
Robinson's criterion yields the crudest possible axiomatization for a
model-completion.  In the case of $\DCF_p$ (again where $p$ is prime or $0$), there are two ways of
refining the axiomatization---refining in the sense of finding seemingly weaker
conditions on models of $\DF_p$ that are still sufficient for being
existentially closed.  It suffices to consider either systems in only one
variable or systems involving only first derivatives.  In the
generalization to several derivations, the former
refinement seems to be of little use; the latter refinement is of use
indirectly, through its introduction of geometric ideas.

\subsection{Single variables}\label{subsect:one-var}

Though the theory \myboxed{\ACF}{} of
algebraically closed fields is the model-completion of the theory of
fields, its axioms (\emph{modulo} the latter theory) can involve only
systems in one variable (indeed, single equations in one variable).
A generalization of this observation is the following, which can be
extracted from the proof of \cite[Thm~17.2, pp.~89--91]{MR0398817}
(see also \cite{MR0491149}):

\begin{lemma}[Blum's Criterion]\label{lem:blum}
  Say $T^*{}_{\forall}\included T\included T^*$.
  \begin{enumerate}
    \item
The theory $T^*$ is the model-completion of $T$ if and only if the
  commutative diagram
  \begin{equation*}
    \xymatrix{
\str M &\\
\str A \ar[u] \ar[r] & \str B \ar@{.>}[ul]
}
  \end{equation*}
of structures and embeddings
can be completed as indicated when
$\str A$ and $\str B$ are models of $T$ and $\str M$ is a
$\size{B}^+$-saturated model of $T^*$.
\item\label{item:blum2}
If $T=T_{\forall}$, it is enough to assume that $\str B$ is generated
over $\str A$
by a single element.
  \end{enumerate}
\end{lemma}

This allows a refinement of Lemma~\ref{lem:rob-crit} in a special
case:

\begin{lemma}
  Suppose $T=T_{\forall}$.  Then Lemma~\ref{lem:rob-crit} still holds
  when $\phi(\tuple x,\tuple y)$ is replaced with
  $\phi(x,\tuple y)$ (where $x$ is a single variable).
\end{lemma}

From Lemma~\ref{lem:blum},
Lenore Blum obtains Theorem~\ref{thm:dcf-1} below in characteristic
$0$, in which case the first two numbered conditions amount to
$K\models\ACF$ (\cite[pp.~298~ff.]{MR0398817} or~\cite{MR0491149}).
If $p>0$, then $\DPF_p$ is not universal, so
part~\eqref{item:blum2} of Blum's criterion does 
\emph{not} apply; Carol Wood instead uses a
primitive-element theorem of Seidenberg \cite{MR0049174} to obtain new
axioms for $\DCF_p$ \cite{MR50:9577}.  These can be combined with
Blum's axioms for $\DCF_0$ to yield the following.  (Here \myboxed{\SCF}{} is
the theory of separably closed fields.)

\begin{theorem}[Blum, Wood]\label{thm:dcf-1}
A model $(K,D)$ of $\DF$ is existentially closed if and only if
\begin{enumerate}
\item
$K\models\SCF$;
  \item
$(K,D)\models\DPF$;
\item
$(K,D)\models\Exists x(f(x,Dx,\dots,D^{n+1}x)=0\land
g(x,Dx,\dots,D^nx)\neq0)$ whenever $f$ and $g$ are ordinary
polynomials over $K$ in tuples $(x^0,\dots, x^{n+1})$ and
$(x^0,\dots,x^n)$ of variables
respectively such that $g\neq0$ and
$\partial f/\partial x^{n+1}\neq0$. 
\end{enumerate}
Hence $\DF$ has a model-companion, called \myboxed{\DCF}.
\end{theorem}

There is a similar characterization of the existentially closed
\emph{ordered} differential fields \cite{MR495120}.

\subsection{First derivatives}\label{sect:first-der}

Alternative simplified axioms for $\DCF$ are parallel to
those found for the model-companion \myboxed{\ACFA}{} of the theory of fields
with an automorphism \cite{MR99c:03046,MR2000f:03109}.  Suppose
$(K,D)\models\DPF$ and $K\models\SCF$.
Given a system
over $(K,D)$, we can rewrite it so that $D$ is applied only to
variables or derivatives of variables; then we can replace each
derivative with a new variable, obtaining a system
\begin{equation}\label{eqn:1st}
  \bigwedge_ff=0\land g\neq0\land D\tuple x=\tuple y,
\end{equation}
where $f,g\in K[\tuple x,\tuple y]$ and $D\tuple x=\tuple y$ stands for $\bigwedge_iDx^i=y^i$.  We can also incorporate this condition into the rest of the system,  writing
$\bigwedge_ff(\tuple x,D\tuple x)=0\land g(\tuple x,D\tuple x)\neq0$. 
Suppose~\eqref{eqn:1st} has the solution $(\tuple a,\tuple b)$.
Then $K(\tuple a,\tuple b)/K$ is separable
\cite[Lem.~1.5, p.~1328]{MR2114160}.  Let
$V$ and $W$ be the varieties over $K$ with generic points $\tuple a$
and $(\tuple a,\tuple b)$ respectively, let \myboxed{T_D(V)}{} be the
twisted tangent bundle of $V$, and let $U$ be the open subset of $W$
defined by the inequation $g\neq0$.
\begin{comment}

 then the situation can be
depicted thus:
\begin{equation*}
\SelectTips{cm}{}
\newdir{ >}{{}*!/-5pt/@{>}}
  \xymatrix{
U \ar@{ >->}[r] \ar@{-->>}[dr]
 & W \ar@{ >->}[r] \ar@{-->>}[d] & T_D(V) \ar@{-->>}[dl]^{(\tuple
  x,\tuple y)\mapsto\tuple x}\\
& V&
}
\end{equation*}

\end{comment}
In characteristic $0$, the model $(K,D)$ of $\DF$ is existentially
closed if and only if, in every such geometric situation, $U$ contains
a $K$-rational point $(\tuple c,D\tuple c)$; this yields the so-called
geometric axioms for $\DCF_0$ found with Anand Pillay
\cite{MR99g:12006}.  In positive characteristic, it is still true
that, if $(\tuple a,\tuple b)$ is a generic point of $V$, then $D$
extends to $K(\tuple a)$ so that $D\tuple a=\tuple b$.  However, an
additional condition is needed to ensure that $D$ extends to all of
$K(\tuple a,\tuple b)$; it is enough to require that the projection of
$T_D(W)$ onto $T_D(V)$ contain a generic point of $W$; this yields
Piotr Kowalski's geometric axioms for $\DCF_p$ \cite{MR2119125}.

By the trick of replacing $g\neq0$ with $z\cdot g=1$, we may
assume that there is no inequation in~\eqref{eqn:1st}.
In an alternative geometric approach to $\DCF$, we can then
consider~\eqref{eqn:1st} as a special case of
\begin{equation}\label{eqn:1-my}
  \bigwedge_ff=0\land\bigwedge_{i<k}Dx^i=g^i, 
\end{equation}
where $f\in K[x^0,\dots,x^{n-1}]$ and $g^i\in
K(x^0,\dots,x^{n-1})$. 
Suppose this has solution $\tuple a$, that is, $(a^0,\dots,a^{n-1})$, which is a generic point of
a variety $V$.  Then we know that $D$ extends so as to map $K(\tuple a)$ into \emph{some} field.  We know too that this extension of $D$ maps the subfield $K(a^0,\dots,a^{k-1})$ of $K(\tuple a)$ into $K(\tuple a)$ itself; indeed, this extension is given by the equations $Da^i=g^i(\tuple a)$.  Now we can extend $D$ further to all of $K(\tuple a)$ so that this becomes a differential field \cite[Lem.~1.2 (3), p.~1326]{MR2114160}.
Going back and picking a new generating tuple for $K(\tuple a)$ as needed,
we may assume that 
$(a^0,\dots,a^{k-1})$ is a separating tran\-scendence-basis of $K(\tuple
a)/K$.  
Then we have a dominant, separable rational map $\tuple
x\mapsto(x^0,\dots,x^{k-1})$ or $\phi$ from $V$ onto $\Aff^k$,
and another rational map $\tuple x\mapsto(g^0(\tuple x),\dots,g^{k-1}(\tuple
x))$ or $\psi$ from $V$ to $\Aff^k$.  So $(K,D)$ is existentially
closed if
and only if $V$ always has a $K$-rational point $P$ such that
$D(\phi(P))=\psi(P)$ \cite[Thm~1.6, p.~1328]{MR2114160}. 

\section{Fields with several derivations}\label{sect:several}

Let \myboxed{\mDF}{} be the theory of fields with $m$ commuting derivations.
Tracey McGrail \cite{MR2001h:03066} axiomatizes the model-completion,
\myboxed{\mDCF_0}, of $\mDF_0$.  Alternative axiomatizations arise as
special cases in 
work of Yoav Yaffe \cite{MR1807840} and Marcus Tressl
\cite{MR2159694}.  There is a common theme:  A differential ideal has
a generating set of a special form; in the terminology of Joseph Ritt
\cite[\S I.5, p.~5]{MR0201431} (when $m=1$) and Ellis Kolchin
\cite[\S I.10, pp.~81~ff.]{MR58:27929}, this is a \emph{characteristic set.}
There is a first-order way to tell, uniformly in the parameters,
whether a given set of differential 
polynomials is a characteristic set of some differential ideal, and
then to tell, if it \emph{is}
a characteristic set, whether it has a root.  In short, the function
$\phi\mapsto\robfn{\phi}$ in Robinson's criterion
(Lemma~\ref{lem:rob-crit}) is defined for
sufficiently many systems $\phi$.  (Applying Blum's criterion, McGrail
and Yaffe consider only systems in one variable, so they must include
inequations in these systems; Tressl uses only equations, in
arbitrarily many variables.)

I do not give the definition of a characteristic set, as not all
ingredients of the definition are needed for the arguments presented
in \S \ref{sect:resolution}.  However, some of the ingredients
\emph{are} needed; these are in~\ref{subsect:terms}.

\subsection{Spaces of derivations}\label{subsect:fga}

In \cite{MR2000487} I attempted to apply the geometric approach
described in~\ref{sect:first-der} to $\mDF_0$.  I worked more
generally with $\DF^m_0$, where \myboxed{\DF^m}{} is the theory of structures
$(K,D_0,\dots,D_{m-1})$ such that $(K,D_i)\models\DF$ for each $i$, and
each bracket $[D_j,D_k]$ is a $K$-linear combination of the $D_i$.
(This is roughly what Yaffe did too.)
In
\cite[\S 2]{MR2114160} I made some minor corrections and otherwise adapted
the argument to arbitrary characteristic.  Nonetheless, in May, 2006,
Ehud Hrushovski showed me a counterexample to \cite[Thm~A,
  p.~926]{MR2000487}, a theorem that was an introductory formulation of
 \cite[Thm~5.7, p.~942]{MR2000487}.  Then I found an error at
 the end of the proof of the latter theorem.  That theorem is simply wrong; the present paper does not so much correct the theorem as replace it.

The developments leading up to the wrong theorem are still of some use.
The general situation is as follows.
Let $(K,D_0,\dots,D_{m-1})\models\DF^m$, and let $\lrv$ be the
$K$-linear span of the $D_i$.  Then $\lrv$ is a Lie-ring, as well as a
vector-space over $K$.  As a vector-space, $\lrv$ has a dual,
\myboxed{\lrv^*}; and there is a derivation \myboxed{\dee}{} from $K$ into
$\lrv^*$ given by
\begin{math}
  D(\dee x)=Dx
\end{math}.
Then $\lrv^*$ has a basis $(\dee t^i\colon i<\ell)$ for some $t^i$ in
$K$ and some $\ell$ no greater than $m$ \cite[Lem.~4.4,
  p.~932]{MR2000487}, and this basis is dual to a basis
$(\partial_i\colon i<\ell)$ of $\lrv$, where 
$[\partial_i,\partial_j]=0$ in each case, and $\dee$ can be given by
\begin{equation}\label{eqn:dee}
  \dee x=\sum_{i<\ell}\dee t^i\cdot\partial_ix
\end{equation}
\cite[Lem.~4.7, p.~934]{MR2000487}.
We can use these ideas to prove the following.

\begin{theorem}\label{thm:if-mDF}
If $\mDF$ has a model-companion, then so does $\DF^m$.
\end{theorem}

\begin{proof}
In characteristic $0$, the result is implicit in \cite[Thm~5.3 and proof]{MR2000487}, explicit in
\cite[\S 3]{MR2286106}; but the proof works generally.  The main point is to find, for any model $(K,D_0,\dots,D_{m-1})$ of $\DF^m$, an extension
in which the named derivations are linearly independent over the larger field.  As above, the space spanned over $K$ by the $D_i$ has a basis $(\partial_i\colon i<\ell)$ of commuting derivations of $K$.
If $\ell<m$, then let $L=K(\alpha^{\ell},\dots,\alpha^{m-1})$,
where
$(\alpha^{\ell},\dots,\alpha^{m-1})$ is algebraically
independent over $K$.
Extend the $\partial_i$ to $L$ so that they
are $0$ at the $\alpha^j$;
then, if $\ell\leq k<m$, define $\partial_k$ to be $0$ on $K$ and to be
$\updelta_k^j$ at $\alpha^j$.  Then $(\partial_i\colon i<m)$ is a
linearly independent $m$-tuple of commuting derivations on $L$; from this, we
obtain linearly independent extensions $\tilde D_i$ of the $D_i$ to $L$ such that
the brackets $[\tilde D_j,\tilde D_k]$ are the same linear combinations of the $\tilde D_i$ that the $[D_j,D_k]$ are of the $D_i$ (by \cite[Lem.~5.2, p.~937]{MR2000487}---or \cite[Lem.~2.1,
  p.~1930]{MR2286106}, by a different method---in characteristic $0$;
generally,  \cite[Lem.~2.4, p.~1334]{MR2114160}).  Then $(K,D_i,\dots,D_{m-1})\included(L,\tilde D_0,\dots,\tilde D_{m-1})$, and the latter is a model of $\DF^m$.  Moreover, $(L,\tilde D_0,\dots,\tilde D_{m-1})$ is an existentially closed model if and only if $(L,\partial_0,\dots,\partial_{m-1})$ is an existentially closed model of $\DF^m$; so a model-companion of $\DF^m$ can be derived from a model-companion of $\mDF$.
\end{proof}

That much stands, and differential forms
are convenient for establishing it.  The theorem, combined with the results of \S \ref{sect:resolution}, will yield a model-companion, $\DCF^m$, of $\DF^m$.

\subsection{A new approach}

In \cite{MR2000487} I tried also to obtain $\DCF_0^m$
independently as follows.  Suppose now we have
a separably closed field~$K$, along with a Lie-ring and
finite-dimensional space $\lrv$ of derivations 
of $K$; as a space, $\lrv$ has (for some $m$) a basis $(\partial_i\colon i<m)$, whose
dual is $(\dee t^i\colon i<m)$, so that the $\partial_i$ commute.  We
may assume  that
$(K,\partial_0,\dots,\partial_{m-1})$ is differentially perfect
\cite[Lem.~2.4]{MR2114160}.
Every system over $(K,\partial_0,\dots,\partial_{m-1})$
 is equivalent to a system of the
form of~\eqref{eqn:1-my}, generalized to
\begin{equation}\label{eqn:m-my}
  \bigwedge_ff=0\land
  \bigwedge_{j<k}\bigwedge_{i<m}\partial_ix^j=g_i^j.
\end{equation}
Here again $f\in K[x^0,\dots,x^{n-1}]$, and $g_i^j\in
K(x^0,\dots,x^{n-1})$. 
By means of~\eqref{eqn:dee}, we can also write the system as
\begin{equation}\label{eqn:many-my}
  \bigwedge_ff=0\land
  \bigwedge_{j<k}\dee x^j=\sum_{i<m}\dee t^i\cdot g_i^j.
\end{equation}
If $(a^0,\dots,a^{n-1})$ or $\tuple a$ is a solution (from some extension),
we may
assume that $(a^0,\dots,a^{k-1})$ is a
separating transcendence-basis of $K(a^0,\dots,a^{k-1})/K$.  However, we can no longer assume that $(a^0,\dots,a^{k-1})$ is a separating transcendence-basis of $K(\tuple a)/K$ itself.  In characteristic $0$, this is shown by the example referred to in \cite[Exple 1.2, p.~927]{MR2000487}; we can adapt the example to positive characteristic $p$ by letting $K=\mathbb F_p(b^{\sigma}\colon\sigma\in\vnn^2)$ and defining $\partial_0b^{(i,j)}=b^{(i+1,j)}$ and $\partial_1b^{(i,j)}=b^{(i,j+1)}$.  Let $(a^1,a^2)$ be algebraically independent over $K$, but $a^0=(a^2)^p$, and define $\partial_0a^0=0=\partial_1a^0$ and $\partial_0a^1=b^{(0,0)}$ and $\partial_1a^1=a^2$.  Then the $\partial_i$ are commuting derivations mapping $K(a^0,a^1)$ into $K(a^0,a^1,a^2)$, and they extend as commuting derivations to the latter field, but not so as to map this field into itself.
This is an important
difference from the case of one derivation; it is what causes the
difficulties in the case of several derivations.  

In our example, $K(a^0,a^1,a^2)/K(a^0,a^1)$ is not separable.
%, so the separating tran\-scendence-basis $(a^0,a^1)$ of $K(a^0,a^1)$ does not extend to a separating transcendence-basis of $K(a^0,a^1,a^2)$.  
However, if we let $a^3$ be a new transcendental and define $\partial_0a^2=b^{(0,1)}$ and $\partial_1a^2=a^3$, then the $\partial_i$ are still commuting derivations, now mapping $K(a^0,a^1,a^2)$ into $K(a^0,a^1,a^2,a^3)$, and the larger field is indeed separable over the subfield.  This will turn out to be possible in general.  That is, in~\eqref{eqn:many-my}, we shall be able to assume that $(a^0,\dots,a^{k-1})$ is a
separating transcendence-basis of $K(a^0,\dots,a^{k-1})/K$ \emph{and} it extends to a separating transcendence-basis $(a^0,\dots,a^{\ell-1})$ of $K(\tuple a)/K$.  This is obvious in characteristic $0$.

The solution $\tuple a$
to~\eqref{eqn:many-my} then can be 
understood as follows.  First
we have the field $K(\tuple a)$, and then~\eqref{eqn:many-my} can be be
written as
\begin{equation}\label{eqn:step-1}
  \bigwedge_{j<k}\dee a^j=\sum_{i<m}\dee t^i\cdot g_i^j(\tuple a).
\end{equation}
A solution of this can be understood as a model
$(L,\tilde{\partial}_0,\dots,\tilde{\partial}_{m-1})$ of $\mDF$
extending $(K,\partial_0,\dots,\partial_{m-1})$ such that $K(\tuple
a)\included L$ and also~\eqref{eqn:step-1} holds when $\dee
a^j=\sum_{i<m}\dee t^i\cdot\tilde{\partial}_ia^j$.  This last condition is 
\begin{equation}\label{eqn:jkim}
\bigwedge_{j<k}\bigwedge_{i<m}\tilde{\partial}_ia^j=g_i^j(\tuple a).  
\end{equation}
Since the
$\tilde{\partial}_i$ commute, it is
necessary that 
\begin{equation}\label{eqn:jkhim}
\bigwedge_{j<k}\bigwedge_{i<m}\bigwedge_{h<i}\tilde{\partial}_h(g_i^j(\tuple
a))=\tilde{\partial}_i(g_h^j(\tuple a))
\end{equation}
\cite[\S 1, p.~926]{MR2000487}.
Any derivative with respect to $\tilde{\partial}_i$ of an element of
$K(\tuple a)$ is a constant plus a linear combination
of the derivatives  $\tilde{\partial}_ia^j$, where $j<\ell$ (by
\cite[Fact~1.1 (0, 2)]{MR2114160}, for example); 
we know
what these derivatives $\tilde{\partial}_ia^j$ are when $j<k$,
by~\eqref{eqn:jkim};
so~\eqref{eqn:jkhim} becomes a linear system over $K(\tuple a)$ in the
unknowns $\tilde{\partial}_ia^j$ where $k\leq j<\ell$.   

If $k=\ell$, then this linear system has no variables, so it is true
or false; its truth is a sufficient condition for~\eqref{eqn:step-1}
to have a solution.  If $k<\ell$, then the linear system is soluble or
not.  If it is soluble, then it is possible to extend the $\partial_i$
to derivations $\tilde{\partial}_i$ as required by~\eqref{eqn:jkim}
that commute on
$K(a^0,\dots,a^{k-1})$; but these derivations need not commute on all
of $K(\tuple a)$.  In \cite{MR2000487} I claimed that they could
commute, and that the solubility 
of~\eqref{eqn:jkhim} was sufficient for solubility
of~\eqref{eqn:step-1} in the sense above.  I was wrong.  

A generic solution to the linear system~\eqref{eqn:jkhim} generates an
extension of $K(\tuple a)$; then we have to check extensibility of the
commuting derivations to \emph{this.}  That is, we are back in the
same kind of situation we started with.  However, it turns out that
there is a bound on the number of times that we need to repeat this
process in order to ensure solubility of the original differential
system.  This is what is shown in \S \ref{sect:resolution} below;
differential forms are apparently not useful for this after all.  

\subsection{A counterexample}
Over a model of $\DF^2$,
let $(a,b,c)$ be an algebraically independent triple.  The counterexample supplied by Hrushovski is the system
\begin{align}\label{eqn:sys-ex}
  \dee a&=\dee t^0\cdot c^2+\dee t^1\cdot c,&
\dee b&=\dee t^0\cdot 2a+\dee t^1\cdot c
\end{align}
(where $c^2$ is the square of $c$; the constants
$(k,\ell)$ of \S \ref{subsect:fga} are now $(2,3)$).
Equivalently, by~\eqref{eqn:dee}, the system comprises the equations
\begin{align*}
  \partial_0a&= c^2,&\partial_1a&=c,&
  \partial_0b&=2a,&\partial_1b&=c.
  \end{align*}
From these, we compute
\begin{align*}
\partial_1\partial_0a&=2c\cdot\partial_1c,&
\partial_0\partial_1a&=\partial_0c,& 
\partial_1\partial_0b&=2\cdot\partial_1a=2c,&
\partial_0\partial_1b&=\partial_0c.
\end{align*}
Equating $\partial_0\partial_1$ and $\partial_1\partial_0$ yields the
linear system
\begin{align}\label{eqn:c}
  \partial_0c-2c\cdot\partial_1c&=0,&
  \partial_0c&=2c,
\end{align}
which has the solution $(\partial_0c,\partial_1c)=(2c,1)$.  But then we
must have
$\partial_1\partial_0c=2\cdot\partial_1c=2$, while
$\partial_0\partial_1c=\partial_01=0$, which means~\eqref{eqn:sys-ex}
has no solution, contrary to my claim in~\cite{MR2000487}.

For the record, the mistake is at the end of the proof of
\cite[Thm~5.7, p.~942]{MR2000487} and can be seen as follows.
Write the system~\eqref{eqn:sys-ex} as 
$\dee a=\alpha$, $\dee b=\beta$; then
\begin{align}\label{eqn:da}
  \dee\alpha
  &=\dee c\wedge(\dee t^0\cdot 2c+\dee t^1),
&&
\begin{split}
\dee\beta
&=\dee a\wedge\dee t^0\cdot 2+\dee c\wedge\dee t^1\\
&=(\dee c-\dee t^0\cdot 2c)\wedge\dee t^1.
\end{split}
\end{align}
Since also $\dee\beta=\dee^2b=0$,
we now have a condition on $\dee c\wedge\dee
t^1$, hence on $\partial_0c$; in particular, $\partial_0c=2c$, which
is what we found above.  But
there is no apparent
condition on $\partial_1c$, so I try introducing a new transcendental,
$d$, for this derivative.  By~\eqref{eqn:c} then,
\begin{equation*}
\dee c=\dee t^0\cdot 2c+\dee t^1\cdot d,
\end{equation*}  
which by~\eqref{eqn:da} yields
\begin{math}
  \dee\alpha
=\dee t^0\wedge\dee t^1\cdot 2c(1-d)
\end{math}.
But we must have $\dee\alpha=0$, so $d=1$, contrary to
assumption.  In short, the next to last sentence of the proof of
 \cite[Thm~5.7]{MR2000487} (beginning `This ideal is linearly disjoint
 from') is simply wrong.  (I had not attempted to argue that it was
 correct.)

\section{Resolution}\label{sect:resolution}

For a correct understanding of the existentially closed differential fields,
it is better not to introduce differential
forms from the beginning, but to allow equations to involve any number
of applications of the derivations.  In
contrast to~\ref{subsect:one-var} above, there does not seem to be an
advantage now in restricting attention to equations in one variable.

\subsection{Terminology}\label{subsect:terms}

I shall now avoid working
with differential polynomials as such, but shall work instead with the
algebraic dependencies that they determine.

Let $(K,\partial_0,\dots,\partial_{m-1})\models\mDF$.  Higher-order
derivatives with respect to the $\partial_i$ can be indexed by
elements of $\vnn^m$:  so, for
$\partial_0{}^{\sigma(0)}\dotsb\partial_{m-1}{}^{\sigma(m-1)}x$, we
may write \myboxed{\partial^{\sigma}x}.
Let \myboxed{\leq}{} be the product ordering of $\vnn^m$.
Then the derivative $\partial^{\sigma}x$ is
\defn{below}{} $\partial^{\tau}x$ (and the latter is \defn{above}{} the
former) if $\sigma\leq\tau$.  (In particular, a derivative is both below and above itself.)  If $n\in\vnn$, then two elements of
$\vnn^n\times n$ will be related by $\leq$ only if they agree in the
last coordinate, so that
\begin{equation*}
  (\sigma,k)\leq(\tau,\ell)\iff\sigma\leq\tau\land k=\ell;
\end{equation*}
we may use the corresponding terminology of `above' and `below', so that
$\partial^{\sigma}x_k$ is below~$\partial^{\tau}x_{\ell}$ if (and only if)
$(\sigma,k)\leq(\tau,\ell)$. 

If
$\sigma\in\vnn^m$, let the sum $\sum_{i<m}\sigma(i)$ be denoted by
\myboxed{\size{\sigma}}: this is the \defn{height}{} of $\sigma$ or of
$\partial^{\sigma}x$.  (Kolchin \cite[\S I.1, p.~59]{MR58:27929} uses
the word \emph{order}.) 
If $n$ is a positive integer, let $\vnn^m\times n$ be (totally) ordered by
\myboxed{\tleq}, which is taken from 
the left lexicographic ordering 
of $\vnn^{m+1}$ by means of the embedding
\begin{equation*}
  (\xi,k)\longmapsto
(\size{\xi},k,\xi(0),\dots,\xi(m-2))
\end{equation*}  
of $\vnn^m\times n$ in $\vnn^{m+1}$.  
Then $(\vnn^m\times n,\tleq)$ is isomorphic to $(\vnn,\leq)$.  
\begin{comment}

In case $n=1$, and $m$ is $2$ or $3$, the picture is thus:
\begin{center}
\hfill
  \begin{pspicture}(-1,-2.5)(3,0.5)
    \psset{radius=0.1}
    \Cnode(0,0){a}
    \uput[ul](0,0){$(0,0)$}
    \Cnode(1,0){b}
    \uput[u](1,0){$(0,1)$}
    \ncline{->}{a}{b}
    \Cnode(0,-1){c}
    \uput[l](0,-1){$(1,0)$}
    \ncline{->}{b}{c}
    \Cnode(2,0){d}
    \uput[ur](2,0){$(0,2)$}
    \ncline{->}{c}{d}
    \Cnode(1,-1){e}
    \uput[dr](1,-1){$(1,1)$}
    \ncline{->}{d}{e}
    \Cnode(0,-2){f}
    \uput[d](0,-2){$(2,0)$}
    \ncline{->}{e}{f}
  \end{pspicture}
\hfill
\begin{pspicture}(-2.5,-0.5)(2.5,3)
    \psset{radius=0.1}
    \Cnode(0,2.4){a}
    \uput[u](0,2.4){$(0,0,2)$}
    \Cnode(0.7,1.2){b}
    \uput[ur](0.7,1.2){$(0,1,1)$}
    \ncline{->}{a}{b}
    \Cnode(1.4,0){c}
    \uput[r](1.4,0){$(0,2,0)$}
    \ncline{->}{b}{c}
    \Cnode(-0.7,1.2){d}
    \uput[ul](-0.7,1.2){$(1,0,1)$}
    \ncline{->}{c}{d}
    \Cnode(0,0){e}
    \uput[d](0,0){$(1,1,0)$}
    \ncline{->}{d}{e}
    \Cnode(-1.4,0){f}
    \uput[l](-1.4,0){$(2,0,0)$}
    \ncline{->}{e}{f}  
\end{pspicture}
\hfill
\mbox{}
\end{center}

\end{comment}
We may
write $(\sigma,k)\tl\infty$ for all $(\sigma,k)$ in $\vnn^m\times n$.
Suppose $(x_h\colon h<n)$ is a tuple of indeterminates.  By ordering
the formal derivatives $\partial^{\sigma}x_k$ in terms of $(\sigma,k)$ and
$\tleq$, we have Kolchin's example of an \emph{orderly ranking} of
derivatives \cite[\S I.8, p.~75]{MR58:27929}.  If
$(\sigma,k)\tl(\tau,\ell)$, I shall say that the derivative
$\partial^{\sigma}x_k$ is
\defn{less}{} than $\partial^{\tau}x_{\ell}$ or is a \defn{predecessor}{}
of $\partial^{\tau}x_{\ell}$, and
 $\partial^{\tau}x_{\ell}$ is \defn{greater}{} than 
 $\partial^{\sigma}x_k$; likewise for the expressions $a_k^{\sigma}$
and $a_{\ell}^{\tau}$, introduced in~\eqref{eqn:L} below.  (So, the
terms just defined refer
to the strict total ordering $\tl$, while `below' and `above' refer to the
partial ordering $\leq$.)

Addition and subtraction on $\vnn$ induce corresponding operations on
$\vnn^m$.  Then 
\begin{gather}\notag
\tau\leq\sigma+\tau,\\\notag
\partial^{\sigma}\partial^{\tau}x_k=
\partial^{\sigma+\tau}x_k,\\\label{eqn:sks}
(\sigma,k)\tleq(\sigma+\tau,k),\\\notag
(\sigma,k)\tleq(\tau,\ell)\Iff 
(\sigma+\rho,k)\tleq(\tau+\rho,\ell).
\end{gather}
If $i<m$, let \myboxed{\ichar}{} denote the
characteristic function of $\{i\}$ in $\vnn^m$, so that
$\partial^{\ichar}=\partial_i$, and more generally
$\partial_i\partial^{\sigma}=\partial^{\sigma+\ichar}$, and also
$\partial_i\partial^{\sigma-\ichar}=\partial^{\sigma}$ if $\sigma(i)>0$. 

Let $L$ be an extension of $K$ with generators that are indexed by an initial
segment of $(\vnn^m\times n,\tleq)$; that is,
\begin{equation}\label{eqn:L}
  L=K(a_h^{\xi}\colon(\xi,h)\tl(\tau,\ell)),
\end{equation}
where $(\tau,\ell)\in\vnn^m\times n$, or possibly $(\tau,\ell)=\infty$,
in which case $L=K(a_h^{\xi}\colon(\xi,h)\in\vnn^m\times n)$.
It could happen that, in the generating tuple $(a_h^{\xi}\colon(\xi,h)\tl(\tau,\ell))$ of $L/K$, the same element of $L$ may appear twice, with different indices.  In this case, when
writing $a_h^{\xi}$, we may mean not just a particular element of $L$, but that element together with the pair $(\xi,h)$ of indices.
For example, by~\eqref{eqn:sks}, if $(\sigma+\ichar,k)\tl(\tau,\ell)$, then $(\sigma,k)\tl(\tau,\ell)$; hence we may say that,
if $a_k^{\sigma+\ichar}$ is one of the generators
of $L/K$, then so is $a_k^{\sigma}$.  Let us say that $L$, with the
tuple of generators given in~\eqref{eqn:L}, meets the \defn{differential condition}{} if there is
no obstacle to extending each derivation $\partial_i$ to a derivation
$D_i$ on $K(a_h^{\xi}\colon(\xi+\ichar,h)\tl(\tau,\ell))$ such that
\begin{equation}\label{eqn:Dia}
  D_ia_k^{\sigma}=a_k^{\sigma+\ichar}
\end{equation}
whenever $(\sigma+\ichar,k)\tl(\tau,\ell)$.
(If the right-hand member of~\eqref{eqn:Dia} is not defined,
then the left need not be defined.)  
To be precise,
if $f$ is a rational function over $K$
in variables $(x_h^{\xi}\colon(\xi,h)\tleq(\sigma,k))$ for some $(\sigma,k)$ in $\vnn^m\times n$, and $D$ is a derivation of $K$, then $f$ has a derivative $Df$, which is the linear function over $K(x_h^{\xi}\colon(\xi,h)\tleq(\sigma,k))$ given by
\begin{equation*}
Df=  \sum_{(\xi,h)\tleq(\sigma,k)}
\frac{\partial f}{\partial
  x_h^{\xi}}\cdot
y_h^{\xi}+ 
f^D.
\end{equation*}
Then the differential condition is that for all such $f$, if
$(\sigma+\ichar,k)\tl(\tau,\ell)$ for some $i$ in $m$, and if
\begin{equation}\label{eqn:fahx}
  f(a_h^{\xi}\colon(\xi,h)\tleq(\sigma,k))=0,
\end{equation}
then
\begin{math}
\partial_if(a_h^{\xi},a_h^{\xi+\ichar}\colon(\xi,h)\tleq(\sigma,k))=0
\end{math}, that is,
\begin{equation}\label{eqn:sumeta}
  \sum_{(\eta,g)\tleq(\sigma,k)}
\frac{\partial f}{\partial
  x_g^{\eta}}(a_h^{\xi}\colon(\xi,h)\tleq(\sigma,k))\cdot
a_g^{\eta+\ichar}+ 
f^{\partial_i}(a_h^{\xi}\colon(\xi,h)\tleq(\sigma,k))=0.
\end{equation}
(Note well the assumption that
 $(\sigma+\ichar,k)\tl(\tau,\ell)$.  In~\eqref{eqn:sumeta}, each
of the $a_g^{\eta+\ichar}$ must exist, even though the
coefficient $(\partial f/\partial
 x_g^{\eta})(a_h^{\xi}\colon(\xi,h)\tleq(\sigma,k))$ might be $0$.)
So the differential condition is \emph{necessary} for the
extensibility of the $\partial_i$ as desired (see for example
\cite[Fact~1.1~(0)]{MR2114160}); sufficiency is part of
Lemma~\ref{lem:above} below.

An extension $(M,D_0,\dots,D_{m-1})$ of
$(K,\partial_0,\dots,\partial_{m-1})$ is \defn{compatible}{} with the
extension $L$ of $K$ given in~\eqref{eqn:L} if $L\included M$,
and~\eqref{eqn:Dia} holds whenever
$(\sigma+\ichar,k)\tl(\tau,\ell)$. 

Borrowing some terminology
used for differential polynomials \cite[\S IX.1, p.~163]{MR0201431},
let us say that a generator $a_k^{\sigma}$ of $L/K$ is a \defn{leader}{} if
it is algebraically dependent over $K$ on its predecessors, that is, 
\begin{equation*}
a_k^{\sigma}\in\alg{K(a_h^{\xi}\colon(\xi,h)\tl(\sigma,k))}.  
\end{equation*}
Then
$a_k^{\sigma}$ is a \defn{separable}{} leader if it is separably
algebraic over $K(a_h^{\xi}\colon(\xi,h)\tl(\sigma,k))$; otherwise, it
is an \defn{inseparable}{} leader.  A separable
leader $a_k^{\sigma}$ is \defn{minimal}{} if
there is no separable leader strictly below it---no separable leader
$a_k^{\rho}$ such that $\rho<\sigma$.

\begin{comment}

For example, in the field $K(a^{(i,j)}\colon(i,j)\leq(3,3))$ depicted
in~\eqref{eqn:matrix} above, the generator
$a^{(3,3)}$ is a (non-minimal) separable leader, and
$a^{(3,1)}$ is a minimal separable leader.  But here we wanted $a^{(3,3)}$
ultimately to be a derivative of $a^{(3,1)}$, namely
 $\partial_1{}^2a^{(3,1)}$.  Passing to a larger field
in~\eqref{eqn:matrix2}, we found the
condition $a^{(3,0)}=a^{(0,2)}$; then $a^{(3,0)}$ became a new separable
leader, strictly below the formerly minimal separable leader $a^{(3,1)}$.

\end{comment}
\begin{lemma}\label{lem:above}
Suppose $(K,\partial_0,\dots,\partial_{m-1})\models\mDF$ and $L$ meets the differential condition, where $L$ is
an extension
$K(a_h^{\xi}\colon(\xi,h)\tl(\tau,\ell))$ of $K$.  Then the derivations
$\partial_i$ extend to derivations $D_i$ from
$K(a_h^{\xi}\colon(\xi+\ichar,h)\tl(\tau,\ell))$ into $L$
 such that~\eqref{eqn:Dia} holds when $(\sigma+\ichar,k)\tl(\tau,\ell)$.
 If $a_k^{\sigma}$ is a
 separable leader, and $(\sigma+\ichar,k)\tl(\tau,\ell)$, then
 \begin{equation}\label{eqn:aksi}
   a_k^{\sigma+\ichar}\in K(a_h^{\xi}\colon (\xi,h)\tl(\sigma+\ichar,k))
 \end{equation}
(that is, $a_k^{\sigma+\ichar}$ is a rational function over $K$ of its
 predecessors);
in particular, $a_k^{\sigma+\ichar}$ is a separable leader.  Therefore
generators of $L/K$ that are above separable leaders are themselves
separable leaders.
\end{lemma}

\begin{proof}
The claim follows from the basic properties of derivations, such as
are gathered in \cite[\S 1]{MR2114160}.
We extend the derivations to each generator in turn, according to the ordering $\tl$.
Suppose
$D_i$ has been defined as desired on $K(a_h^{\xi}\colon(\xi,h)\tl(k,\sigma))$ (so that~\eqref{eqn:Dia} holds
whenever it applies).
If $a_k^{\sigma}$ is not a leader, then we are free to define the derivatives $D_ia_k^{\sigma}$ as we like. Now suppose
$a_k^{\sigma}$ is a separable leader, and
$(\sigma+\ichar,k)\tl(\tau,\ell)$.  Then $D_ia_k^{\sigma}$ is
obtained by differentiating the minimal polynomial of $a_k^{\sigma}$
over $K(B)$.  That is, $D_ia_k^{\sigma}$ is obtained by
differentiating an equation like~\eqref{eqn:fahx}; by the differential
condition, $D_ia_k^{\sigma}$ must be $a_k^{\sigma+\ichar}$ as given
by~\eqref{eqn:sumeta}; this shows that $a_k^{\sigma+\ichar}$ is a
rational function over $K$ of its predecessors.  

Finally, in a
positive characteristic~$p$,
$a_k^{\sigma}$ may be an inseparable leader.  Then $(a_k^{\sigma})^{p^r}\in \sep{K(a_h^{\xi}\colon
  (\xi,h)\tl(\sigma,k))}$ for some positive $r$.  If
$(\sigma+\ichar,k)\tl(\tau,\ell)$, then we are free to define
$D_ia_k^{\sigma}$ as $a_k^{\sigma+\ichar}$, provided
$D_i((a_k^{\sigma})^{p^r})=0$.  But again this condition is ensured by
the differential condition.  Indeed, we may suppose~\eqref{eqn:fahx}
shows the separable dependence of $(a_k^{\sigma})^{p^r}$ over the
predecessors of $a_k^{\sigma}$.  
That is, we can understand $f(a_h^{\xi}\colon(\xi,h)\tleq(\sigma,k))$
as $g((a_k^{\sigma})^{p^r})$ for some separable polynomial $g$ over
$K(a_h^{\xi}\colon(\xi,h)\tl(\sigma,k))$. 
Then $D_i((a_k^{\sigma})^{p^r})$ is
obtained from~\eqref{eqn:sumeta}, provided we replace the term
$(\partial f/\partial
x_k^{\sigma})(a_h^{\xi}\colon(\xi,h)\tleq(\sigma,k))\cdot
a_k^{\sigma+\ichar}$ with
$g'((a_k^{\sigma})^{p^r})\cdot
D_i((a_k^{\sigma})^{p^r})$.  But in the present case, the former term
is $0$.  Since,
after the replacement, the resulting 
equation still holds, we must have $D_i((a_k^{\sigma})^{p^r})=0$.
\end{proof}

\subsection{A solubility condition}

If $(K,\partial_0,\dots,\partial_{m-1})\models\mDF$, then this model
has an extension whose underlying field is the separable closure of
$K$ (as by \cite[Lem.~3.4, p.~930]{MR2000487} and \cite[Lem.~2.4,
  p.~1334]{MR2114160}).  We shall need this in a more general form:

\begin{lemma}\label{lem:comm}
  Suppose a field $M$ has two subfields $L_0$ and $L_1$, which in turn
  have a common subfield $K$.  For each $i$ in $2$, suppose there is a
  derivation $D_i$ mapping $K$ into $L_i$ and $L_{1-i}$ into $M$.
  Then the bracket $[D_0,D_1]$ is a well-defined
  derivation on $K$.  Suppose it is the $0$-derivation.
\begin{comment}

, that is, the
  following diagram commutes.
  \begin{equation*}
    \begin{CD}
      K @>{D_1}>> L_0\\
@V{D_0}VV @VV{D_0}V\\
L_1 @>>{D_1}> M
    \end{CD}
  \end{equation*}

\end{comment}
Suppose also that $a$ is an element of $M$ that is separably algebraic over
$K$.  Then each $D_i$ extends uniquely to $K(a)$, and $D_ia\in
L_{1-i}(a)$, so $D_{1-i}D_ia$ is also well-defined.  Moreover, $[D_0,D_1]a=0$.
\end{lemma}

\begin{proof}
Obvious from \cite[Fact~1.1~(2)]{MR2114160}.
\begin{comment}

  The claim follows from standard facts, at least if $L_0=K=L_1$; but
  the proof is the same in the general case.  Indeed, though the
  derivations $D_0$ and $D_1$ are defined on $K$, their bracket
  $[D_0,D_1]$ need not be so, since the compositions $D_0D_1$ and
  $D_1D_0$ need not be so; but if they are, then $[D_0,D_1]$ is a
  \emph{derivation} on $K$.
  A derivation on $K$ extends uniquely to $\sep K$; if the derivation
  is $0$ on $K$, then it is $0$ on $\sep K$
  \cite[Fact~1.1~(2)]{MR2114160}.  
In the present case, as $a\in \sep K$, so $D_ia\in L_{1-i}(a)$, and
  therefore $D_ia\in\sep{L_{1-i}{}}$; hence $D_{1-i}D_ia$ is defined.
Thus $[D_0,D_1]$ is
  defined on $K(a)$, where $a\in \sep K$; and if the bracket is $0$ on
  $K$, then is $0$ at $a$.

  \end{comment}
\end{proof}

In a positive characteristic, the possibility of inseparably algebraic
extensions presents a challenge, which however is handled by the following.
\begin{comment}

The following example illustrates how the challenge
can be overcome.  Over a differential field
$(K,\partial_0,\partial_1)$, where $\Char K=p>0$,
the differential equation
\begin{equation*}
  (\partial_0x)^p+x=\partial_1x
\end{equation*}
determines an extension
$K(a^{\xi}\colon\size{\xi}\leq 2)$ of $K$ that meets the differential
condition: the generators form a triangle thus: 
\begin{equation*}
  \begin{matrix}
    a&b^p+a&b^p+a\\
    b&b    &     \\
    c&     &
  \end{matrix}
\end{equation*}
where $(a,b,c)$ is algebraically independent over $K$.
In particular, $a^{(1,0)}$ (which has the value $b$) is an inseparable leader,
but none of the generators of height $2$ (namely, $a^{(0,2)}$,
$a^{(1,1)}$ and $a^{(2,0)}$) is an inseparable 
leader.

\end{comment}

\begin{theorem}\label{thm:first}
Suppose $(K,\partial_0,\dots,\partial_{m-1})\models\mDF$, and $K$
has an extension $K(a_h^{\xi}\colon\size{\xi}\leq2r\land
h<n)$ meeting the differential condition for some positive integers
$r$ and $n$.  Suppose further that, whenever $a_k^{\sigma}$ is a
minimal separable leader, 
then $\size{\sigma}\leq r$.  Then  $(K,\partial_0,\dots,\partial_{m-1})$
has an extension 
$(M,D_0,\dots,D_{m-1})$
compatible with
$K(a_h^{\xi}\colon\size{\xi}<2r\land h<n)$.  
\end{theorem}

\begin{proof}
  The claim can be compared to and perhaps derived from a
  differential-alge\-braic lemma of 
  Rosenfeld \cite[\S I.2]{MR0107642}, at least in
  characteristic~$0$.  Here I give an independent 
  argument, for arbitrary characteristic.  
We shall obtain $M$ recursively as
  $K(a_h^{\xi}\colon(\xi,h)\in\vnn^m\times n)$, at the same time
  proving inductively that the $\partial_i$ can be extended to $D_i$ so
  that~\eqref{eqn:Dia} holds in all cases.

Let $L=K(a_h^{\xi}\colon\size{\xi}<2r\land h<n)$; this is
$K(a_h^{\xi}\colon(\xi,h)\tleq((2r-1,0,\dots,0),n-1))$.  Then
by~\eqref{eqn:sumeta}, the differential
condition requires of the tuple 
$(a_h^{\xi}\colon\size{\xi}=2r\land h<n)$ only that it solve
some linear equations over $L$.  The hypothesis of our claim is that
there \emph{is} a solution, namely
 $(a_h^{\xi}\colon\size{\xi}=2r\land h<n)$.  We may therefore
assume that this tuple is a
\emph{generic} solution of these equations.  In particular, no entry
of this tuple is an inseparable leader.  (If, instead of being chosen generically, the entries of
$(a_h^{\xi}\colon\size{\xi}=2r\land h<n)$ were chosen from the field $L$,
then this field
would be closed under the desired extensions $D_i$ of $\partial_i$,
and the derivations $D_i$ would commute on the subfield
$K(a_h^{\xi}\colon\size{\xi}+1<2r\land h<n)$; but they might not
commute on all of $L$.) 

Now, as an inductive hypothesis, suppose we have the extension
$K(a_h^{\xi}\colon(\xi,h)\tl(\tau,{\ell}))$ of $K$ meeting the differential
condition, so that there are derivations $D_i$ as given by
Lemma~\ref{lem:above}; suppose also that
\begin{enumerate}
\item
if $a_k^{\sigma}$ is a minimal separable leader, then
$\size{\sigma}\leq r$;
\item
if $a_k^{\sigma}$ is an inseparable leader, then
$\size{\sigma}<2r$.
\end{enumerate}
We need
to choose $a_{\ell}^{\tau}$ in such a way that these conditions still
hold for $K(a_h^{\xi}\colon(\xi,h)\tleq(\tau,{\ell}))$.  
The inductive
hypothesis is correct when $\size{\tau}\leq2r$, and then the desired
conclusion follows; so we may
assume $\size{\tau}>2r$.  
Hence, if $\tau(i)>0$, so that $\tau-\ichar$ is defined, then
$\size{\tau-\ichar}\geq2r$, so
$a_{\ell}^{\tau-\ichar}$ is not an inseparable leader.  

If $a_{\ell}^{\tau-\ichar}$ is not a
leader at all, for any $i$ in $m$, then we may let $a_{\ell}^{\tau}$ be a
new transcendental, and we may define each derivative
$D_ia_{\ell}^{\tau-\ichar}$ as this \cite[Fact~1.1~(1)]{MR2114160}.  

In the other case, $a_{\ell}^{\tau-\ichar}$ is a
separable leader for some $i$.  Then $D_ia_{\ell}^{\tau-\ichar}$ is determined
(Lemma~\ref{lem:above}).  We want to let $a_{\ell}^{\tau}$ be this
derivative.  However, possibly also $a_{\ell}^{\tau-\jchar}$ is a
separable leader, where $i\neq j$.  In this case, we must check that
\begin{equation}\label{eqn:comm}
  D_ja_{\ell}^{\tau-\jchar}=D_ia_{\ell}^{\tau-\ichar},
\end{equation}
that is, $[D_i,D_j]a_{\ell}^{\tau-\ichar-\jchar}=0$. 

There are minimal separable leaders $a_{\ell}^{\pi}$ and $a_{\ell}^{\rho}$
below $a_{\ell}^{\tau-\ichar}$ and
$a_{\ell}^{\tau-\jchar}$ respectively.
Let $\nu$ be $\pi\vee\rho$,
the least upper bound of $\{\pi,\rho\}$ with respect to $\leq$.
Then $\nu\leq\tau$.  But
$\size{\nu}\leq\size{\pi}+\size{\rho}\leq2r<\size{\tau}$; so $\nu<\tau$. 
Hence $\nu\leq\tau-\kchar$ for some $k$ in $m$, which means
$a_{\ell}^{\nu}$ is below $a_{\ell}^{\tau-\kchar}$.
Consequently,
\begin{enumerate}
  \item
$a_{\ell}^{\pi}$ is below both $a_{\ell}^{\tau-\ichar}$ and
    $a_{\ell}^{\tau-\kchar}$;
\item
$a_{\ell}^{\rho}$ is below both $a_{\ell}^{\tau-\jchar}$ and
    $a_{\ell}^{\tau-\kchar}$.
\end{enumerate}
If $k=j$, then $a_{\ell}^{\pi}$ is below
$a_{\ell}^{\tau-\ichar-\jchar}$, so this is a separable leader.  As
$D_i$ and $D_j$ commute on
$K(a_h^{\xi}\colon (\xi,h)\tl(\tau-\ichar-\jchar,{\ell}))$ by the 
differential condition, they must
commute also at $a_{\ell}^{\tau-\ichar-\jchar}$ (Lemma~\ref{lem:comm}),
so~\eqref{eqn:comm} is established.  The argument is the same if
$k=i$.  If $k$ is different from $i$ and $j$, then again the same
argument yields
 $D_ja_{\ell}^{\tau-\jchar}=D_ka_{\ell}^{\tau-\kchar}$ and
 $D_ka_{\ell}^{\tau-\kchar}=D_ia_{\ell}^{\tau-\ichar}$,
so~\eqref{eqn:comm} holds.   

In no case did we introduce a new minimal separable leader or an
  inseparable leader.  This completes the induction and the
  proof.  
\end{proof}

%The claim at the end of the last subsection (\ref{subsect:example}) is now justified.

In terms of differential polynomials and ideals, the theorem can be
understood as follows.  Given the hypothesis of the theorem, let $S$
be the set of differential polynomials
$f(\partial^{\xi}x_h\colon\size{\xi}<2r\land h<n)$, where $f$
ranges over the ordinary polynomials over $K$ such that
$f(a^{\xi}_h\colon\size{\xi}<2r\land h<n)=0$.  Then $S$
includes a characteristic set for the differential ideal that it
generates.  

We can now characterize the existentially closed models of $\mDF$ by
means of the following lemma.  The lemma follows from unproved
statements in \cite[\S 0.17, p.~49]{MR58:27929}; let's just prove it here.  

\begin{lemma}\label{lem:fin}
For every $m$ in $\vnn$ and positive integer $n$, 
  every antichain of $(\vnn^m\times n,\leq)$ is finite.
\end{lemma}

\begin{proof}
The general case
follows from the case when $n=1$, since if $S$ is an antichain of
$(\vnn^m\times n,\leq)$, then 
\begin{equation*}
S=\bigcup_{j<n}\{(\xi,h)\in S\colon h=j\}, 
\end{equation*}
and each component of the union is in bijection with
an antichain of $(\vnn^m,\leq)$.  As an inductive hypothesis,
suppose every antichain of 
$(\vnn^{\ell},\leq)$ is finite; but suppose also, if possible, that
there is an infinite 
antichain $S$ of $(\vnn^{\ell+1},\leq)$.  Then $S$ contains some
$\sigma$.  By inductive hypothesis, the subset
\begin{equation*}
  \bigcup_{j\leq\ell}\bigcup_{i\leq\sigma(j)}\{\xi\in S\colon\xi(j)=i\}
\end{equation*}
of $S$
is a finite union of finite sets, so its complement in
$S$ has infinitely many elements $\tau$; but then $\sigma<\tau$, so
$S$ was not an antichain.
\end{proof}

\begin{theorem}\label{thm:dcf}
Suppose $(K,\partial_0,\dots,\partial_{m-1})\models\mDF$.  Then the
following are equivalent:
\begin{enumerate}
  \item\label{item:ec}
The model
$(K,\partial_0,\dots,\partial_{m-1})$ of $\mDF$ is existentially closed.
\item\label{item:<}
For all positive integers $r$ and $n$, if $K$ has an extension
$K(a_h^{\xi}\colon\size{\xi}\leq2r\land h<n)$ meeting the
differential condition such that $\size{\sigma}\leq r$ whenever
$a_k^{\sigma}$ is a minimal separable leader, then the tuple
$(a_h^{\xi}\colon\size{\xi}<2r\land h<n)$ has a
specialization $(\partial^{\xi}b_h\colon\size{\xi}<2r\land
h<n)$ for some tuple $(b_h\colon h<n)$ of elements of $K$. 
\end{enumerate}
\end{theorem}

\begin{proof}
Assume~\eqref{item:ec} and the hypothesis of~\eqref{item:<}.  Let $S$ be a 
  (finite) generating set of the ideal of
 $(a_h^{\xi}\colon\size{\xi}<2r\land h<n)$ over $K$.  By
  Theorem~\ref{thm:first}, the system 
  \begin{equation*}
    \bigwedge_{f\in S}f(\partial^{\xi}x_h\colon\size{\xi}<2r\land h<n)=0
  \end{equation*}
has a solution in some extension, hence it has a solution
in $K$ itself, which means the conclusion of~\eqref{item:<} holds.
So~\eqref{item:<} is necessary for~\eqref{item:ec}.

Every system over $(K,\partial_0,\dots,\partial_{m-1})$ is equivalent to a
system of equations.  Suppose such a system has a
solution $(a_h\colon h<n)$ in some extension $(L,D_0,\dots,D_{m-1})$.  
Then the extension $K(\partial^{\xi}a_h\colon(\xi,h)\in(\vnn^m\times
n))$ has a \emph{finite} set of minimal separable leaders, by
Lemma~\ref{lem:fin}, since this
set is indexed by an antichain of $(\vnn^m\times n,\leq)$.
Hence there is $r$ large enough that all of these minimal separable
leaders are also generators of 
$K(D^{\xi}a_h\colon\size{\xi}\leq r\land h<n)$.  We may assume also
that $r$ is large enough that $\size{\sigma}\leq r$ for every
derivative $\partial^{\sigma}x_k$ that appears in the original
system.  The hypothesis of~\eqref{item:<} is now satisfied when each
$a^{\sigma}_k$ is taken as $D^{\sigma}a_k$.  If the conclusion
of~\eqref{item:<} follows, then $(b_h\colon h<n)$ is a solution of the
original system.
Thus,~\eqref{item:<} is sufficient
for~\eqref{item:ec}. 
\end{proof}

\begin{corollary}\label{cor:dcf}
  The theory $\mDF$ has a model-companion, $\mDCF$.
\end{corollary}

\begin{proof}
Let $(K,\partial_0,\dots,\partial_{m-1})$ be a model of $\mDF$, let $L$ be an extension $K(a_h^{\xi}\colon\size{\xi}\leq2r\land
h<n)$ of $K$  meeting the differential condition, and suppose
$\size{\sigma}\leq r$ whenever $a_k^{\sigma}$ is a minimal separable
leader.  
That is, assume the hypothesis of Condition~\eqref{item:<} of
 the theorem. 
Write $\tuple a$ for $(a_h^{\xi}\colon\size{\xi}<2r\land h<n)$ and $\tuple b$ for $(a_h^{\xi}\colon\size{\xi}=2r\land h<n)$, so that $L=K(\tuple a,\tuple b)$.
The ideal of $K[\tuple x,\tuple y]$ comprising the polynomials that are $0$ at
$(\tuple a,\tuple b)$ is generated by a set $\{f(\tuple p,\tuple
x,\tuple y)\colon f\in
T\}$, where $T$ is a finite subset of $\mathbb Z[\tuple z,\tuple
  x,\tuple y]$, and $\tuple p$ is a (finite) list of parameters from
$K$.  We may assume that $T$ has a subset $S$, where $S\included\mathbb Z[\tuple z,\tuple x]$ and $\{g(\tuple p,\tuple x)\colon g\in S\}$ generates the ideal of $\tuple a$.
  We need to ensure that there is a formula $\phi(\tuple z)$ such
that 
\begin{enumerate}
\item
 $\phi(\tuple p)$ holds in $(K,\partial_0,\dots,\partial_{m-1})$;
 \item
 for \emph{every} model $(K,\partial_0,\dots,\partial_{m-1})$ of
 $\mDF$ and every tuple $\tuple q$ from $K$, if $\phi(\tuple q)$ holds
 in the model, then the polynomials $f(\tuple q,\tuple x,\tuple y)$ (where $f\in T$)
 generate over $K$ a prime ideal, a generic zero of which generates an
 extension of $K$ as in the hypothesis of Condition~\eqref{item:<} of
 the theorem. 
\end{enumerate}
In this case, by the theorem, $\mDCF$ will have, as axioms, the axioms of $\mDF$, along with one
sentence of the form 
\begin{equation*}%\label{eqn:ax}
\phi(\tuple z)\lto\Exists{\tuple x}\Bigl(\bigwedge_{g\in
  S}g(\tuple z,\tuple x)=0\land
\bigwedge_{i<m}\bigwedge_{h<n}\bigwedge_{\size{\xi+\ichar}<2r}x_h^{\xi+\ichar}
=\partial_ix_h^{\xi}\Bigr). 
\end{equation*}
for each model $(K,\partial_0,\dots,\partial_{m-1})$ of $\mDF$ and
each extension $K(a_h^{\xi}\colon\size{\xi}\leq2r\land h<n)$ as above.

So now we must show that the formula $\phi$ exists as desired; that
is, we must show that there are first-order conditions on the parameters
$\tuple p$ as required.  This we can do as follows.

That the ideal $I$ generated by some finite set of polynomials is a prime
ideal---this condition is a first-order condition on the parameters of
the polynomials, by van den Dries and
Schmidt~\cite{MR739626}.  (As they point
  out, the results of theirs 
  that we shall use are not original with them; the
  proofs are original.)
In particular, there is some $N$ depending only on the degrees of the
generating polynomials and their numbers of variables such that if $fg\in
I\lto f\in I\lor g\in I$ for all polynomials $f$ and $g$ of degree
less than $N$, then the implication holds for all $f$ and $g$. 

Here $I$ may be the ideal of $(\tuple a, \tuple b)$ or of $\tuple a$, defined in terms of $T$ or $S$ as above.
The extension $K(\tuple a,\tuple b)$ meets the differential condition,
because the derivatives of each $g(\tuple p,\tuple x)$ (where $g\in S$) are certain
combinations (which can be made explicit) of the $f(\tuple p,\tuple x,\tuple y)$ (where $f\in T$).
Thus there is a \emph{sufficient}
first-order condition on the parameters $\tuple p$ for the meeting of the
differential condition; and $\tuple p$ does meet this condition. 

Alternatively, if a polynomial $F$
and finitely many additional  polynomials $G$ with
parameters $\tuple p$ are given, the condition that $F$ be a member of
the ideal $I$ generated by the $G$ is a first-order condition on
$\tuple p$.  This follows from the existence of a uniform bound on the
degrees of the polynomial coefficients needed to obtain $F$ from the
polynomials $G$ if indeed $F\in I$.  This bound is uniform in the
sense that it depends only on the degrees of $F$ and the $G$ and the
number of their variables.  The existence of this bound is again shown
by van den Dries and
Schmidt~\cite{MR739626}.

Moreover, for a list $\tuple w$ of variables that appear in the
polynomials $G$, the condition that no non-zero polynomial in $\tuple
w$ alone belongs to $I$ is
also a first-order condition on $\tuple p$.  Indeed, since this
condition is invariant under replacement of the underlying field by
its algebraic closure, we may appeal to the general result that
Morley rank is definable in algebraically closed fields and more
generally in strongly minimal sets~\cite[\S6.2]{MR1924282}.  For some
formula $\psi$, the condition holds if and only if $\psi(\tuple p)$ is
true in the algebraic closure of the underlying field; but by
quantifier-elimination in algebraically closed fields, we may
assume also $\psi$ is quantifier-free, so $\psi(\tuple p)$ is true in
the algebraic closure of a field if and only if it is true in the
field itself.

For each leader in $(\tuple a,\tuple b)$, we may assume that some irreducible
polynomial in $T$ shows that it is a leader.  The
irreducibility of this polynomial is a first-order condition on $\tuple p$
(since in general 
primeness of ideals is first-order).  The leader is separable if and only if
the formal derivative of its irreducible polynomial is not zero, that
is, not all of its coefficients 
belong to the ideal generated by $\{f(\tuple p,\tuple x,\tuple y):f\in T\}$; as noted above, this is a first-order
condition.  The condition that an entry in $(\tuple a,\tuple b)$ is
\emph{not} a leader at all is also a first-order condition on the
parameters, since this condition is just that no non-zero polynomial
in certain variables belongs to the ideal generated by $\{f(\tuple p,\tuple x,\tuple y)\colon f\in T\}$. 

Now we can arrange that $\phi(\tuple z)$ establishes all of the
conditions discussed; and this is enough. 
\end{proof}

By Theorem~\ref{thm:if-mDF}, $\DF^m$ now also has a model-companion.

\subsection{Differential forms again}

The condition in Theorem~\ref{thm:first} can be adjusted to yield the
following: 

\begin{theorem}\label{thm:second}
Suppose $(K,\partial_0,\dots,\partial_{m-1})\models\mDF$, and $K$
has an extension $K(a_h^{\xi}\colon\size{\xi}\leq\size{\mu}\land
h<n)$ meeting the differential condition for some $\mu$ in $\vnn^m$
and some positive integer $n$.  Suppose further that, if $a_k^{\sigma}$ is a
minimal separable leader, 
then $\sigma\leq\mu$.  Then  $(K,\partial_0,\dots,\partial_{m-1})$
has an extension 
compatible with
$K(a_h^{\xi}\colon\size{\xi}<\size{\mu}\land h<n)$.  
\end{theorem}

\begin{proof}
  The proof is as for Theorem~\ref{thm:first}, \emph{mutatis
  mutandis.}  What needs adjusting is the choosing of
  $a_{\ell}^{\tau}$ in case both $a_{\ell}^{\tau-\ichar}$ and
  $a_{\ell}^{\tau-\jchar}$ are separable leaders.  Again we have
  minimal separable leaders
 $a_{\ell}^{\pi}$ and $a_{\ell}^{\rho}$
below $a_{\ell}^{\tau-\ichar}$ and
$a_{\ell}^{\tau-\jchar}$ respectively.  Since we may assume
  $\size{\mu}<\size{\tau}$, there is some $k$ in $m$
  such that $\mu(k)<\tau(k)$.  If $k=j$, then
  $\pi(j)\leq\mu(j)<\tau(j)$, so
  $\pi(j)\leq(\tau-\jchar)(j)=(\tau-\ichar-\jchar)(j)$.  Then
  $\pi\leq\tau-\ichar-\jchar$, so $a_{\ell}^{\pi}$ is below
  $a_{\ell}^{\tau-\ichar-\jchar}$.  Now we can proceed as before. 
\end{proof}

As Theorem~\ref{thm:first} yields Theorem~\ref{thm:dcf}, so
Theorem~\ref{thm:second} yields a characterization of the existentially
closed models of $\mDF$.  Moreover,
Theorems~\ref{thm:first} and~\ref{thm:second} can be combined in the
following way:

\begin{theorem}\label{thm:third}
Suppose $(K,\partial_0,\dots,\partial_{m-1})\models\mDF$, and $K$
has an extension $K(a_h^{\xi}\colon\size{\xi}\leq2r\land
h<n)$ meeting the differential condition for some positive integers $n$
and $r$.  Suppose further that, for each $k$ in $m$, either
 $\size{\sigma}\leq r$ whenever $a_k^{\sigma}$ is a 
minimal separable leader, or else there is some $\tau$ in $\vnn^m$
such that $\size{\tau}=2r$, and
 $\size{\sigma}\leq\size{\tau}$ whenever $a_k^{\sigma}$ is a 
minimal separable leader.
Then  $(K,\partial_0,\dots,\partial_{m-1})$
has an extension 
compatible with
$K(a_h^{\xi}\colon\size{\xi}<2r\land h<n)$.    
\end{theorem}

\begin{proof}
  Combine the proofs of
Theorems~\ref{thm:first} and~\ref{thm:second}.
\end{proof}

There is a corresponding first-order characterization of the
models of $\mDCF$, parallel to Theorem~\ref{thm:dcf} and
Corollary~\ref{cor:dcf}.

\subsection{Another sufficient condition}

If $(K,\partial_0,\dots,\partial_{m-1})\models\mDF$, and
$K(a_h^{\xi}\colon\size{\xi}\leq\size{\pi}\land h<n)$ is an extension $L$
of $K$ meeting the differential condition, this by itself is not enough
to ensure that $(K,\partial_0,\dots,\partial_{m-1})$
has an extension compatible with
$L$.  However, if such
an extension does exist, then its existence can be shown by means of 
Theorem~\ref{thm:first}, provided $\size{\pi}$ can be made large
enough:  this is
Theorem~\ref{thm:s} below, which relies on the existence of
bounds as in the following.

\begin{lemma}
  For all positive integers $m$ and $n$, for all sequences $(a_i\colon
  i\in\vnn)$ of positive integers, there is a 
  bound on the length of strictly increasing chains
  \begin{equation}\label{eqn:chain}
    S_0\pincluded S_1\pincluded S_2\pincluded\dotsb
  \end{equation}
of antichains $S_k$ of $(\vnn^m\times n,\leq)$, where also
$S_k\included\{(\xi,h)\colon\size{\xi}\leq a_k\}$.
\end{lemma}

\begin{proof}
Divide and conquer.
  First reduce to the case when $n=1$.  Indeed, suppose the claim
  does hold in this case.  Suppose also, as an inductive hypothesis,
  that the claim holds when $n=\ell$.  Now fix $m$ and the sequence
  $(a_i\colon i\in\vnn)$ or rather  $(a(i)\colon i\in\vnn)$, and
  consider arbitrary chains as 
  in~\eqref{eqn:chain}, where $n=\ell+1$.  Analyze each $S_k$ as
  $S_k'\cup S_k''$, where
  \begin{align*}
    S_k' &=\{(\xi,h)\in S_k\colon h<\ell\},&
    S_k''&=\{(\xi,h)\in S_k\colon h=\ell\}.
  \end{align*}
For each $k$ such that $S_{k+1}$ exists, at least one of the
inclusions $S_k'\included S_{k+1}'$ and $S_k''\included S_{k+1}''$
is strict; also, by our assumption, there is an upper bound
$f(k)$ on those $r$ such that
\begin{equation}\label{eqn:sk''}
  S_k''\pincluded S_{k+1}''\pincluded\dotsb\pincluded S_{r-1}''.
\end{equation}
The function $f$ depends only on
$m$ and $(a_i\colon i\in\vnn)$), not on the choice
of chain in~\eqref{eqn:chain}.

Let $k(0)=0$, and if $k(i)$ has been chosen, let $k(i+1)$ be the least
$r$, if it exists, such that $S_{k(i)}'\pincluded S_r'$.  
Here $k(i)$ does depend on the chain.
But if $r$ is maximal in~\eqref{eqn:sk''}, and $S_r'$ exists,
then $S_k'\pincluded S_r'$.
Hence
$k(i+1)\leq f(k(i))$.  Since the function $f$ is not necessarily
increasing, we derive from it the increasing function $g$, where
$g(k)=\max_{i\leq k}f(i)$.  Then $x\leq y\implies g(x)\leq g(y)$, so
\begin{equation}\label{eqn:krg}
  k(r)\leq f(k(r-1))\leq g(k(r-1))\leq g\circ g(k(r-2))\leq\dotsb
\leq
\overbrace{g\circ\dotsb\circ g}^r(0)=g^r(0).
\end{equation}
In particular,
$S_{k(r)}\included\{(\xi,h)\colon\size{\xi}\leq
a(g^r(0))\}$.  The sequence $(a(g^i(0))\colon i\in\vnn)$ does not
depend on the original chain.
Hence the inductive hypothesis applies to the chain
\begin{equation}\label{eqn:chain2}
  S_{k(0)}'\pincluded S_{k(1)}'\pincluded\dotsb,
\end{equation}
showing
that there is $s$ (independent of the original chain) such that
$k(s)$ is defined, and $r\leq s$ for all entries $S_{k(r)}'$
in~\eqref{eqn:chain2}.  
Hence also, by~\eqref{eqn:krg}, if
$S_r'$ is an entry in~\eqref{eqn:chain2}, then $r\leq k(s)\leq
g^s(0)$. 

Now suppose $S_r'$ is the final entry in~\eqref{eqn:chain2}.  Then
$S_r''\pincluded S_{r+1}''\pincluded\dotsb$; but if $S_t''$ is an
entry of this chain, then $t<f(r)\leq g(r)\leq g(g^s(0))=g^{s+1}(0)$. 

Therefore the
original chain in~\eqref{eqn:chain} has a final entry $S_t$, where
$t<g^{s+1}(0)$.  Thus the claim holds when $n=\ell+1$.  By induction, the
claim holds for all positive $n$, provided it holds when $n=1$.

It remains to show that, for all positive $m$, for all sequences
$(a_i\colon i\in\vnn)$, there is a bound on the length of chains
\begin{equation}\label{eqn:chain3}
  S_0\pincluded S_1\pincluded S_2\pincluded\dotsb
\end{equation}
of antichains $S_k$ of $(\vnn^m,\leq)$,
where $S_k\included\{\xi\colon\size{\xi}\leq a_k\}$.  The claim is
trivially true when $m=1$.  Suppose it is true when $m=\ell$.  Now let
$m=\ell+1$, and suppose we have a chain as in~\eqref{eqn:chain3}.  We
may assume that $S_0$ contains some $\sigma$.  If $i<m$ and
$j\in\vnn$, let
\begin{equation*}
  S_k^{i,\,j}=\{\xi\in S_k\colon\xi(i)=j\}.
\end{equation*}
Then the inductive hypothesis applies to chains of the form
\begin{equation*}
  S_{k(0)}^{i,\,j}\pincluded
  S_{k(1)}^{i,\,j}\pincluded
  S_{k(2)}^{i,\,j}\pincluded\dotsb.
\end{equation*}
Moreover, if $\tau\in S_k$, then $\tau(i)\leq\sigma(i)$ for some $i$
in $m$ (since
$\sigma$ is also in $S_k$, and this is an antichain).  Hence
\begin{equation*}
  S_k=\bigcup_{i<m}\bigcup_{j\leq\sigma(i)}S_k^{i,\,j},
\end{equation*}
a union of no more than $\size{\sigma}+m$-many sets, hence no more than
$a_0+m$-many 
sets.  So the proof can proceed as in the reduction to $n=1$: for each
$k$ such that $S_{k+1}$ exists, one of the inclusions
$S_k^{i,\,j}\included S_{k+1}^{i,\,j}$ is strict, and so forth.
\end{proof}

\begin{theorem}\label{thm:s}
  Suppose $m$, $r$,
  and $n$ are positive integers.  Then there is a positive integer $s$, where $r\leq s$, such that, if 
 $(K,\partial_0,\dots,\partial_{m-1})\models\mDF$, and
$K(a_h^{\xi}\colon\size{\xi}\leq s\land
  h<n)$ meets the differential condition, then
  $(K,\partial_0,\dots,\partial_{m-1})$ has an extension that is
  compatible with $K(a_h^{\xi}\colon\size{\xi}\leq
  r\land h<n)$.
\end{theorem}

\begin{proof}
Suppose $K(a_h^{\xi}\colon\size{\xi}\leq 2^tr\land h<n)$ 
meets the differential condition for some~$t$.  When $u\leq t$, let
  $K_u=K(a_h^{\xi}\colon\size{\xi}\leq 2^ur\land h<n)$, and let $S_u$
  be the set of minimal separable leaders of $K_u$.  Then we have an
  increasing chain $S_0\included S_1\included\dots\included S_t$.
  By the preceding lemma, there is a value of $t$, depending only on
  $m$, $r$, and $n$, large
  enough that this chain cannot be strictly increasing.  Then
  $S_u=S_{u+1}$ for some $u$ less than this~$t$.  
  Then $K_{u+1}$ satisfies the hypothesis of Theorem~\ref{thm:first}.  
So $(K,\partial_0,\dots,\partial_{m-1})$ has an extension compatible
  with $K(a_h^{\xi}\colon\size{\xi}<2^{u+1}r\land h<n)$, and \emph{a
  fortiori} with 
 $K(a_h^{\xi}\colon\size{\xi}\leq r\land h<n)$.  In short, the desired
  $s$ is $2^tr$.
\end{proof}

This theorem yields yet another first-order characterization of the models of
$\mDCF$.

%\bibliographystyle{amsplain}
%\bibliographystyle{asl}
%\bibliography{../references}

%\bibliography{references}

\def\rasp{\leavevmode\raise.45ex\hbox{$\rhook$}} \def\cprime{$'$}
  \def\cprime{$'$} \def\cprime{$'$} \def\cprime{$'$}
\providecommand{\bysame}{\leavevmode\hbox to3em{\hrulefill}\thinspace}
\providecommand{\MR}{\relax\ifhmode\unskip\space\fi MR }
\providecommand{\MRhref}[2]{%
  \href{http://www.ams.org/mathscinet-getitem?mr=#1}{#2}
}
\providecommand{\href}[2]{#2}

\end{document}